\documentclass[a4paper, reqno]{amsart}
\usepackage[margin=3.5cm]{geometry}

\usepackage{latexsym}
\usepackage{amsmath}
\usepackage{epsf}
\usepackage[latin1]{inputenc}
\usepackage{graphics, a4wide}
\usepackage{enumerate}
\usepackage{calc}
\usepackage{amssymb}
\usepackage[active]{srcltx}
\usepackage{epsfig}
\usepackage{pdfsync}

\usepackage{color}

\usepackage{amsthm}
\usepackage{amscd}
\usepackage{esint}

\renewcommand{\S}{\mathbb{S}}
\newcommand{\K}{\mathbb{K}}

\newcommand{\p}{\varphi}
\newcommand{\la}{\lambda}

\newcommand{\D}{\nabla}

\def\d{\partial}

\def\e{\varepsilon}

\def\nat{\rm nat}

\def\weak{\rightharpoonup}
\newcommand{\R}{\mathbb{R}}

\def\CC{\mathbb{C}}

\def\FF{\mathcal{F}}

\def\DD{\mathbb{D}}

\def\eps{\varepsilon}
\def\D{\nabla}

\def\XXint#1#2#3{{\setbox0=\hbox{$#1{#2#3}{\int}$}
     \vcenter{\hbox{$#2#3$}}\kern-.5\wd0}}

\newcounter{bei}

\renewcommand{\t}{\widetilde}
\renewcommand{\o}{\overline}

\numberwithin{equation}{section}

\theoremstyle{plain}
\begingroup
\newtheorem{theorem}{Theorem}[section]
\newtheorem{lemma}[theorem]{Lemma}
\newtheorem{proposition}[theorem]{Proposition}

\endgroup

\theoremstyle{definition}
\begingroup

\newtheorem{remark}[theorem]{Remark}
\newtheorem{example}[theorem]{Example}
\endgroup

\DeclareMathOperator{\Tr}{tr}

\DeclareMathOperator{\sym}{sym}

\DeclareMathOperator{\sgn}{sgn}

\author[L. Freddi]{Lorenzo Freddi}
\author[P. Hornung]{Peter Hornung}
\author[M.G. Mora]{Maria Giovanna Mora}
\author[R. Paroni]{Roberto Paroni}

\address[L. Freddi]{DIMA,
via delle Scienze 206, 33100 Udine, Italy
}\email{lorenzo.freddi@uniud.it}
\address[P. Hornung]{Fachrichtung Mathematik,
TU Dresden,
01062 Dresden,
Germany
}\email{peter.hornung@tu-dresden.de}
\address[M.G. Mora]{Dipartimento di Matematica, Universit\`a di Pavia, via Ferrata 1, 27100 Pavia, Italy}
\email{mariagiovanna.mora@unipv.it}
\address[R. Paroni]{DADU, Universit\`{a} degli Studi di
Sassari, Palazzo del Pou Salit,
07041 Alghero (SS), Italy} \email{paroni@uniss.it}

\begin{document}
\title[A variational model for anisotropic and naturally twisted ribbons]{A variational model for anisotropic \\ and naturally twisted ribbons}

\date{}

\maketitle

\begin{abstract}
We consider thin plates whose energy density is a quadratic function of the difference between the second fundamental form of the deformed configuration and a ``natural'' curvature tensor. This tensor either denotes the second fundamental form of the stress-free configuration, if it exists, or a target curvature tensor. In the latter case, residual stress arises from the geometrical frustration involved in the attempt to achieve the target curvature: as a result, the plate is naturally twisted, even in the absence of external forces or prescribed boundary conditions.
Here, starting from this kind of plate energies, we derive a new variational one-dimensional model for naturally twisted ribbons by means of $\Gamma$-convergence. Our result generalizes, and corrects, the classical Sadowsky energy to geometrically frustrated anisotropic ribbons with a narrow, possibly curved, reference configuration.
\end{abstract}

\section{Introduction}\label{Sec0}

Ribbons are ubiquitous in the physical world \cite{AT11, AEK11,CBK93, KTM12, PS10}. Recently, they have received a great deal of attention. This is true, in particular, for M\"obius strips and helical bands,
\cite{AgDeKo, BH2015, Ef2015, ESK11, FF16,  KF2015, SH2015, To2015}.
This renewed interest is also due to their manifold potential applications,
which range from physics/electro-technology to chemistry/nano-technology \cite{FKS12, He12, RP13, SB09, SH2007, TTO02, TTT05}.

Geometrically a ribbon is a strip of thickness $h$,
width $\e$, and centerline length $\ell$, with $h\ll\e\ll\ell$. Because of anisotropic pre-strains, inhomogeneous swelling, plastic deformations or differential growth, ribbons may not have a stress-free configuration.
Hyper-elastic theories for these bodies have been recently formulated in terms of deformations that are measured with respect to a reference metric rather than a reference configuration \cite{ESK09, ESK13}.

Several plate models for these materials have been obtained by studying the $\Gamma$-limit of various scalings of the energy, as $h$ goes to zero. In particular, in~\cite{FrJaMoMu, LP2011,S2007} the energy density of the deduced model  is a quadratic function of the difference between
the second fundamental form
of the deformed configuration and a ``natural'' curvature tensor. This  tensor either denotes
the second fundamental form of the natural (stress-free) configuration or a target curvature tensor.
In the latter case, residual stress arises from the geometrical frustration involved in the attempt to achieve the target curvature: as a result, the ribbon is naturally twisted, even in the absence of external forces or prescribed boundary conditions. By controlling the ``natural'' curvature tensor one may select the shape spontaneously attained by the ribbon: this is the focus of several studies aimed at designing new structures \cite{AgDeKo, KES2007, GMG14, SYU11, TV13, ZAL99}.

Given that also $\e\ll\ell$, after having let $h$ go to zero, it is interesting to find one-dimensional models that characterize very narrow strips, by considering the limit as $\e$ tends to zero.
A limit energy for  homogeneous, isotropic, elastic ribbons with a rectangular stress-free configuration  was put forward by
Sadowsky~\cite{Sadowsky1}, see~\cite{HF2015b} for a recent English translation. This energy, now known as the Sadowsky energy, depends on the curvature and torsion of the centerline of the band and it is singular at the points where the curvature vanishes. A formal justification of the Sadowsky energy was given by Wunderlich \cite{Wunder,To2015}. Only very recently, in \cite{FrHoMoPa}, it has been proved by means of  $\Gamma$-convergence that the Sadowsky energy is correct for  ``large'' curvature of the centerline of the strip, while for ``small'' curvature the correct limit energy  is significantly different from the Sadowsky energy. We shall further address this point at the end of the introduction.

Before discussing the contents of our paper we mention that one-dimensional models could be obtained from the three-dimensional theory also by letting $h$ and $\e$ go to zero simultaneously.
Within the non-linear theory of elasticity for homogeneous bodies with a stress-free configuration several limit energies, corresponding to different scalings, have been obtained in \cite{FrMoPa,FrMoPa2}.

In the geometrically frustrated setting, one-dimensional models have been formally deduced from two-dimensional models in~\cite{DA2015, KES2007, GMG14, SH2007}  by following the procedure of Wunderlich \cite{Wunder,To2015}.

In this paper we consider a two-dimensional energy that coincides with that obtained in~\cite{S2007} by letting $h$ go to zero (see also~\cite{FrJaMoMu, LP2011}) and the same problem considered in \cite{DA2015} but with more general symmetries. We assume the reference configuration to be given by a sequence of two-dimensional ``thin'' regions parametrized by $\e$. These regions are not necessarily rectangular, they may have a curved centerline and a smoothly varying width. The admissible deformations are isometries
and their energy depends quadratically on the difference between the second fundamental form of the deformed configuration and a ``natural'' curvature tensor. By letting the parameter $\eps$ go to zero, under appropriate assumptions on the limit behaviour of the ``natural'' curvature tensor, we identify the $\Gamma$-limit of the (suitably re-scaled) sequence of energy functionals in a topology that ensures compactness of the sequence of minimizers.

Our result not only provides a rigorous derivation of the energy of a very narrow ribbon, but also corrects several formal justifications that are found in the literature. In addition, we allow the energy density to be anisotropic: an intrinsic anisotropy and not simply the one scattered
by the presence of  the ``natural'' curvature tensor as in \cite{Ch2014,ChMaSrHa2012, GiMa2012}.
Limit models within this generality, as far as we know, have not been deduced, not even formally.
We also prove a relaxation result for quadratic functionals with a determinant constraint (see Section \ref{sec_relax}), that is of interest in its own right
and is a fundamental ingredient to deal with the nonlinear isometry constraint in weak topologies.

The limit energy that we deduce depends on three vector fields (directors) $d_1$, $d_2$, and $d_3$,  where $d_1$ is tangent to the limit deformation, $d_2$ represents the ``transversal'' orientation of the strip, and $d_3$ is orthogonal to $d_1$ and $d_2$. The system of directors may not be orthonormal; in fact, they are related to the geometry
of the reference configuration  by means of a covariant basis $D=(D_1,D_2)$ through the constraints
$$
d_\alpha\cdot d_\beta=D_\alpha\cdot D_\beta ,\qquad
d_1'\cdot(d_3\wedge d_1)=D_1'\cdot(e_3\wedge D_1).
$$
The first constraint implies that the ribbon is unsherable and inextensible, while the second constraint is a consequence of the intrinsic nature of the geodesic curvature.
The energy functional is then given by
$$
J(d_1,d_2,d_3)=\int_{-\ell/2}^{\ell/2} \o Q(x_1,d_1'\cdot d_3, d_2'\cdot d_3)\,ds,
$$
where $\ell$ is the length of the centerline of the strip. The quantities $d_1'\cdot d_3$ and $d_2'\cdot d_3$ are usually called, within the theory of rods, bending and twisting, respectively. Denoting the energy density of the plate by $Q$, the limit energy density $\o Q$ is defined in two steps:
first, two positive constants $\alpha_\K^+$ and $\alpha_\K^-$ are defined by
$$
\alpha_\K^\pm:=\sup\{ \alpha>0: \ Q(M)\pm\alpha\det M\geq0 \text{ for every } M\in \R^{2\times 2}_{\sym}\},
$$
and then the energy density $\overline Q$ is given by
\begin{align*}
\overline Q(x_1,\mu,\tau) : = \min \Big\{  \big(Q(M-& D^{-T}A^{\circ}D^{-1})+\alpha^+_\K{(\det M)^+}+\alpha^-_\K {(\det M)^-}\ \big)\det D:\\
&M=\mu D^1\otimes D^1+\tau(D^1\otimes D^2+D^2\otimes D^1)+\gamma D^2\otimes D^2, \ \gamma\in\R
\Big\},
\end{align*}
where $(D^1,D^2)$ denote the contravariant basis in the reference configuration, i.e.,
$D^\alpha\cdot D_\beta=\delta_{\alpha\beta}$, while $A^\circ=A^\circ(x_1)$ characterizes the limit behaviour of the ``natural'' curvature tensor, and $(\det M)^\pm$ denote the positive and negative part of $\det M$.

In the very particular case considered by Sadowsky~\cite{Sadowsky1, HF2015b} and Wunderlich \cite{Wunder,To2015}, which corresponds to $Q(M)=|M|^2$, $A^\circ=0$, and $D$ equal to the identity, the energy density reduces to
$$
\overline Q(x_1,\mu,\tau)  =
\left\{
\begin{array}{ll}
 \displaystyle \frac{(\mu^2+\tau^2)^2}{\mu^2} &
  \mbox{if }  \mu^2 >  \tau^2,\\[2ex]
 4\tau^2 & \mbox{if }  \mu^2 \le  \tau^2,
\end{array}
\right.
$$
and coincides with that found in \cite{FrHoMoPa}.
If $\mu$ and $\tau$ are interpreted as the curvature and the torsion of the centerline of the band,
this function agrees with the Sadowsky energy density only in the regime $\mu^2>\tau^2$; this is the ``large'' curvature regime to which we alluded earlier in
the introduction.

The paper is organized as follows. In Section \ref{sec_energy} we introduce the sequence of energy
functionals and in Section~\ref{sec_rescaling} we rescale them on a fixed domain.
In Section~\ref{sec_comp} we study the compactness properties of sequences with bounded
energy and state the $\Gamma$-convergence result. Section~\ref{sec_relax} is devoted to the relaxation of quadratic
functionals with a constraint on the determinant. This result is the crucial ingredient for the identification of the correct $\Gamma$-limit and is used in the proof of both the liminf and the limsup inequality. The construction of the recovery sequence also requires several geometric and approximation results for isometric immersions, that are proved in Section~\ref{Pro}. Finally,
in Section~\ref{proofs} we prove the $\Gamma$-convergence result.

\section{The energy of an inextensible elastic ribbon}\label{sec_energy}

We consider an inextensible elastic ribbon whose configurations in the three-dimensional space are isometric to a planar region $S_\eps$, where  $\eps>0$ is a small parameter. The region $S_\eps\subset \R^2$ will be taken as reference configuration and its geometry will be specified below.
Any smooth deformation $u:S_\eps\to\R^3$ will satisfy the isometry constraint
\begin{equation}
  \label{isou}
(\D u)^T(\D u) = I,
\end{equation}
where $I$ denotes the $2\times 2$ identity matrix. In coordinates,
\eqref{isou} reads
$\partial_\alpha u\cdot\partial_\beta u = \delta_{\alpha\beta}$.
We denote by
$$
\nu_u = \d_1 u\wedge\d_2 u
$$
the unit normal to $u$, and by
$A_u : S_\e\to\R^{2\times 2}_{\sym}$
the second fundamental form of $u$. It is defined by
 $(A_u)_{\alpha\beta} := \nu_u\cdot\d_\alpha\d_\beta u$ or, equivalently, by $A_u:=(\nabla^2u_i)(\nu_u)_i$.

We assume the energy density of the strip to be quadratic and to depend on the second fundamental form,
but we neither assume the material to be isotropic nor the reference configuration to be  stress free.
Let $A^{\nat}_\eps\in L^2(S_\eps;\R^{2\times 2}_{\rm sym})$ be a symmetric tensor field that either represents the second fundamental form of a ``natural'' configuration or a target curvature tensor field not necessarily corresponding to a configuration (this latter case is usually addressed as non-Euclidean ribbons).
The bending rigidity is taken into account by a linear map $\K$ from
$\R^{2\times2}_{\rm sym}$ into itself. We assume $\K$ to be symmetric,
i.e., $\K A\cdot B=\K B \cdot A$ for every $A, B\in \R^{2\times 2}_{\rm sym}$. Moreover, we assume $\K$ to be positive definite, i.e.,
there exists a  constant $c>0$ such that
$\K A\cdot A\ge c|A|^2$
for every $A\in \R^{2\times 2}_{\rm sym}$.

The energy of the ribbon takes the form
$$
E_{\e}(u) = \frac{1}{2\e}\int_{S_{\e}} \K (A_u(x)-A^{\nat}_\eps(x))\cdot(A_u(x)-A^{\nat}_\eps(x))\, dx.
$$
Its domain of definition is the set of deformations
$u\in W^{2,2}(S_\e;\R^3)$ that satisfy the
constraint~\eqref{isou}.

\subsection*{The region $\boldsymbol{S_\eps}$}

To define the region $S_\eps$ we introduce the rectangle
$\Omega_{\e} = I\times (-\e/2, \e/2)$, where
$I$ denotes the interval $(-\ell/2, \ell/2)$ with $\ell>0$.
Then
$$
S_\eps=\chi(\Omega_\eps)
$$
where $\chi:\R^2\to\R^2$  is an injective orientation preserving map of class $C^2$. We assume  that
$$
|\partial_1\chi|(x_1,0)=1\quad \forall\,x_1\in\R,
$$
so that the length of the curve $\chi(\{x_2=0\})$ in $S_\eps$  is also equal to $\ell$.

Set $\Omega := I\times (-1/2, 1/2)$ and let $\rho_\eps:\Omega\to\Omega_\eps$ be defined by $\rho_\eps(x):=(x_1,\e x_2)$.
We define
$$
D^\e:=(\nabla \chi)\circ \rho_\e,
$$
and
$$
D^\e_\alpha:=D^\e e_\alpha=(\partial_\alpha\chi)\circ \rho_\e.
$$
The pair of vectors $D^\e_1$ and $D^\e_2$ is the covariant basis in the
reference configuration.

For later use we note that there exists a constant $c>0$ such that
\begin{equation}\label{boundsDe}
c\le\det D^\e(x)\le \frac{1}{c},\quad c\le|D^\e(x)| \le \frac{1}{c}\quad \mbox{ for every }x\in\Omega,
\end{equation}
and that
\begin{equation}\label{dedu}
D^\e\to \nabla\chi(\cdot,0)=:D
\end{equation}
uniformly. We set $D_\alpha:=D e_\alpha$ and remark that $|D_1|=1$.

\section{The rescaled bending energy}\label{sec_rescaling}

Let $\chi_\e:\Omega\to S_\e$ be the function
$$
\chi_\e:=\chi\circ\rho_\e
$$
that maps the fixed rectangular region into the reference configuration.

\vspace{1ex}

\begin{center}
  \includegraphics[width=8cm]{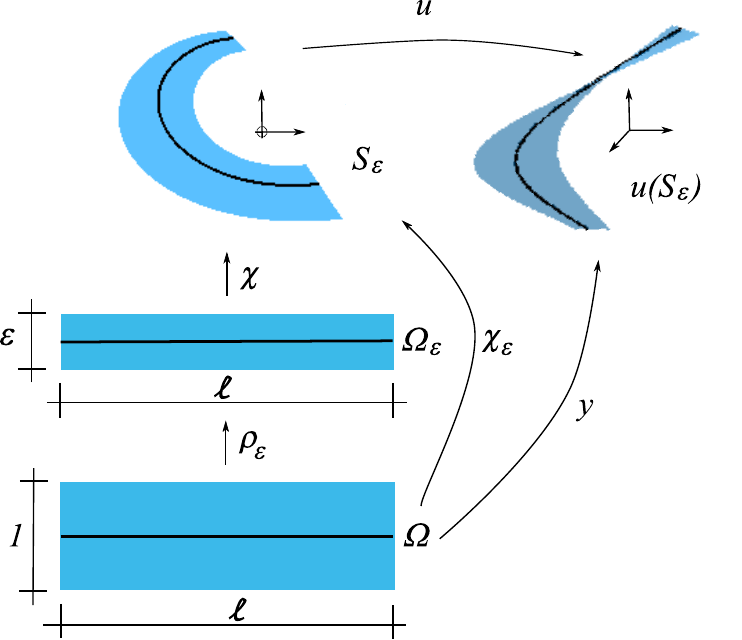}
\end{center}

%\begin{center}
%  \includegraphics[width=10cm]{Sepsilon}
%\end{center}

\vspace{1ex}

Setting
$$
R^\e:=\nabla\rho_\e=e_1\otimes e_1+\e e_2\otimes e_2,
$$
we have
$
\nabla\chi_\e=D^\e R^\e.
$
With a given deformation $u:S_\e\to\R^3$ we associate a rescaled deformation $y: \Omega\to\R^3$ by setting
$$
y:= u\circ\chi_\e.
$$
Then $\nabla y=(\nabla u)\circ \chi_\e\nabla\chi_\e$,
which can be rewritten in terms of the directors of the reference configuration as
\begin{equation}\label{dyu}
\partial_1 y=(\nabla u)\circ \chi_\e D_1^\e,\qquad
\frac{\partial_2 y}{\e}=(\nabla u)\circ \chi_\e D_2^\e.
\end{equation}
As $u$ satisfies \eqref{isou}, we immediately deduce that
\begin{eqnarray}\label{ISO}
&\partial_1 y\cdot\partial_1 y=D^\e_1\cdot D^\e_1,&\nonumber\\
&\displaystyle\partial_1 y\cdot\frac{\partial_2 y}{\e}=D^\e_1\cdot D^\e_2,&\\
&\displaystyle\frac{\partial_2 y}{\e}\cdot\frac{\partial_2 y}{\e}=D^\e_2\cdot D^\e_2.&\nonumber
\end{eqnarray}
Thus, if  $u\in W^{2,2}(S_\e;\R^3)$ satisfies \eqref{isou},
then the rescaled deformation $y$ belongs to the space
$$
W^{2,2}_{{\rm iso},\eps}(\Omega;\R^3):= \{ y\in W^{2,2}(\Omega;\R^3): y \mbox{ satisfies }\eqref{ISO} \text{ a.e.\ in } \Omega\}.
$$

Let
$$
n_y:=\nu_u\circ\chi_\e=\frac{\partial_1 y\wedge\e^{-1}\partial_2 y}{|\partial_1 y\wedge\e^{-1}\partial_2 y |}
$$
denote the unit normal to $y$. The second fundamental forms of $u$
and $y$ are related by
$$
\nabla^2y_i (n_{y})_i=\nabla\chi_\e^T A_u\circ \chi_\e \nabla\chi_\e.
$$
From this identity we deduce
\begin{equation}
  \label{AyAu}
A_u\circ\chi_\e= (D^\e)^{-T} A_{y,\e}(D^\e)^{-1}
\end{equation}
where
$$
A_{y,\e}:=(R^\e)^{-1}\nabla^2y_i(n_y)_i (R^\e)^{-1}
$$
is the rescaled second fundamental form of $y$. This can be rewritten in a more explicit form as
$$
A_{y,\e} =
n_{y}\cdot\d_1\d_1 y \,e_1\otimes e_1 +n_{y}\cdot\frac{\d_1\d_2 y}{\e}
(e_1\otimes e_2+e_2\otimes e_1)+
n_{y}\cdot\frac{\d_2\d_2 y}{\e^2} e_2\otimes e_2.
$$
The energy in terms of the rescaled deformation is given by
$J_\e:W^{2,2}_{{\rm iso},\eps}(\Omega;\R^3)\to [0,+\infty)$, defined as
$$
J_\e(y)=\frac{1}{2}\int_{\Omega} \K\, (D^\e)^{-T}(A_{y,\eps}-A^\circ_\eps)(D^\e)^{-1}\cdot(D^\e)^{-T}(A_{y,\eps}-A^{\circ}_\eps)(D^\e)^{-1}
\det D^\e\, dx,
$$
where we have set
\begin{equation}\label{Ao}
A^\circ_\e:=(D^\e)^T A^{\nat}_\e\circ \chi_\e D^\e.
\end{equation}
We note that the relation between the bending energy and the rescaled energy is $J_\e(y)=E_\eps(u)$.

\section{Compactness and $\Gamma$-limit}\label{sec_comp}

Hereafter, we assume that
\begin{equation}\label{wconvAo}
A^\circ_\e \to A^\circ\qquad \mbox{in }L^2(\Omega; \R^{2\times 2}),
\end{equation}
with $A^\circ=A^\circ(x_1)$, that is $A^\circ\in  L^{2}(I; \R^{2\times 2}_{\sym})$.

\begin{lemma}\label{compactness}
Let $(y_{\e})\subset W^{2,2}_{{\rm iso},\e}(S;\R^3)$ be a sequence of scaled isometries such that
\begin{equation}
\label{cp}
\limsup_{\e\to 0} J_{\e}(y_{\e}) < \infty.
\end{equation}
Then, up to a subsequence and additive constants, there exist a deformation $y\in W^{2,2}(I;\R^3)$ and
three vector fields $d_1,d_2\in W^{1, 2}(I;\R^3)$ and $d_3:=(d_1\wedge d_2)/|d_1\wedge d_2|$
satisfying
\begin{eqnarray}
\label{constraints}
&d_1 = y', \quad d_\alpha\cdot d_\beta=D_\alpha\cdot D_\beta ,& \\
%\mbox{and}\quad d_1'\cdot d_2 = 0 \quad \mbox{ a.e.\ in }I,
\label{constraints-2}
&d_1'\cdot(d_3\wedge d_1)=D_1'\cdot(e_3\wedge D_1),\label{geodesic} &
\end{eqnarray}
almost everywhere in $I$,
such that
\begin{equation}
\label{conv}
y_{\e}\weak y \mbox{ in }W^{2,2}(\Omega; \R^3),\quad
\partial_1 y_{\e} \weak d_1\mbox{ in }W^{1,2}(\Omega; \R^{3}),
\quad
\frac{\partial_2 y_{\e}}{\e} \weak d_2\mbox{ in }W^{1,2}(\Omega; \R^{3}),
\end{equation}
and $A_{y_{\e},\e} \weak A$ in $L^{2}(\Omega; \R^{2\times 2}_{\sym})$, where
\begin{equation}
\label{conv_second}
A=d_1'\cdot d_3\, e_1\otimes e_1 +d_2'\cdot d_3\,
(e_1\otimes e_2+e_2\otimes e_1)+
\gamma\, e_2\otimes e_2,
\end{equation}
with $\gamma\in L^2(\Omega)$.
\end{lemma}

\begin{remark}
Lemma \ref{compactness} naturally extends to a more intrinsic
setting, where the deformation $v := u\circ\chi$
is considered as the natural variable and the energy is defined on the class
of isometric immersions of the surface $\Omega_{\e}$ endowed with a given Riemannian metric
$g$ (which in the present case coincides with $(\D\chi)^T(\D\chi)$). In this setting
formulae \eqref{constraints} and \eqref{constraints-2} follow from the continuity of $g$ and
of the metric connection (Christoffel symbols) defined by $g$.
A similar remark applies to Theorem~\ref{Gamma}--(i) below.
Details on this general approach will be given in the forthcoming paper~\cite{FrHoMoPa2}.
\end{remark}

\begin{proof}[Proof of Lemma~\ref{compactness}.]
Let $(y_{\e})\subset W^{2,2}_{{\rm iso}, \e}(S;\R^3)$ be a sequence satisfying
\eqref{cp}. Then, by using the fact that $\K$ is positive definite and \eqref{boundsDe}, we find
\begin{align*}
C &> c\int_{\Omega} |(D^\e)^{-T}(A_{y_\e,\e}-A^\circ_\eps)(D^\e)^{-1}|^2
\det D^\e\, dx\\
&\ge c \int_{\Omega} |A_{y_\e,\e}-A^\circ_\eps|^2/|D^\e|^4\, dx\ge c \int_{\Omega} |A_{y_\e,\e}-A^\circ_\eps|^2\, dx,
\end{align*}
where the second inequality holds since there exists a constant $c>0$ such that $|BAC|\ge c |A|/(|B^{-1}|\, |C^{-1}|)$ for every matrix $A$ and any invertible matrices $B$ and $C$.
Thus, from \eqref{wconvAo} it follows that
\begin{equation}\label{BoundAye}
\limsup_{\e\to 0} \|A_{y_\e,\e}\|_{L^2(\Omega)} <+\infty.
\end{equation}
Also, combining the fact that $y_{\e}\in W^{2,2}_{{\rm iso}, \e}(\Omega;\R^3)$ with \eqref{boundsDe} gives the bound
\begin{equation}\label{l211}
\limsup_{\e\to 0} ( \|\d_1 y_\e\|_{L^\infty(\Omega)}+\|\e^{-1}\d_2 y_\e\|_{L^\infty(\Omega)})<+\infty.
\end{equation}

We now show that
\begin{equation}\label{l210}
\limsup_{\e\to 0} \big(\|\d_1\d_1 y_{\e}\|_{L^2(\Omega)} + \|\e^{-1}\d_1\d_2 y_{\e}\|_{L^2(\Omega)} + \|\e^{-2}\d_2\d_2 y_{\e}\|_{L^2(\Omega)} \big) < +\infty.
\end{equation}
To prove this it is convenient to set
$$
d_1^\e=\d_1 y_\e, \quad d_2^\e= \frac{\d_2 y_\e}{\e},\quad
d_\e^1=- \frac{n_{y_\e}\wedge d_2^\e}{|d_1^\e\wedge d_2^\e|},\quad %\mbox{ and }
d_\e^2= \frac{n_{y_\e}\wedge d_1^\e}{|d_1^\e\wedge d_2^\e|}.
$$
Since $d_\alpha^\e=(\nabla u)\circ \chi_\e D_\alpha^\e$, see \eqref{dyu}, and $u$ is an isometry, we have that
$|d_1^\e\wedge d_2^\e|=|D_1^\e\wedge D_2^\e|$. Thus, from \eqref{boundsDe} and
\eqref{l211} we deduce that
$$
\limsup_{\e\to 0} ( \|d_1^\e\|_{L^\infty(\Omega)}+\|d_2^\e\|_{L^\infty(\Omega)}+\|d^1_\e\|_{L^\infty(\Omega)}+\|d^2_\e\|_{L^\infty(\Omega)})<+\infty.
$$
Moreover, since $d_\alpha^\e\cdot d^\beta_\e=\delta_{\alpha\beta}$, we have
\begin{align}
\d_1\d_1 y_{\e}&=(\d_1\d_1 y_{\e}\cdot \d_1 y_\e) d^1_\e+(\d_1\d_1 y_{\e}\cdot \e^{-1}\d_2 y_\e) d^2_\e+(\d_1\d_1 y_{\e}\cdot n_{y_\e}) n_{y_\e}\nonumber\\
&=(\d_1D_1^\e\cdot D_1^\e) d^1_\e+[\d_1(D_1^{\e}\cdot D_2^{\e})- (2\e)^{-1}\d_2(D_1^{\e}\cdot D_1^{\e})] d^2_\e+(e_1\cdot A_{y_\e,\e}e_1) n_{y_\e},\label{partialy11}
\end{align}
where the second equality follows from \eqref{ISO} and the definition of $A_{y_\e,\e}$.
Since
$$
\e^{-1}\d_2(D_1^{\e}\cdot D_1^{\e})=\e^{-1}\d_2[(\d_1\chi\cdot \d_1\chi)\circ \rho_\e]=[\d_2(\d_1\chi\cdot \d_1\chi)]\circ \rho_\e,
$$
it follows that the first two terms on the right-hand side of \eqref{partialy11} are uniformly bounded in $L^\infty$, while the third is bounded in $L^2$ by \eqref{BoundAye}.
We have therefore proved that $\limsup_{\e\to 0} \|\d_1\d_1 y_{\e}\|_{L^2(\Omega)}<+\infty$.
The other two bounds appearing in \eqref{l210} are proven similarly.

From \eqref{l211} and \eqref{l210} we infer that, up to additive constants, the sequence $(y_\e)$ is uniformly bounded in $W^{2,2}(\Omega;\R^3)$. Therefore, up to subsequences,
we have that $y_\e\weak y$ in $W^{2,2}(\Omega;\R^3)$ and strongly in $W^{1,p}(\Omega;\R^3)$ for every $p<\infty$.
Inequality \eqref{l211} imply that $y$ is independent of $x_2$. The convergence just stated also implies that $\d_1 y_\e\weak d_1$
weakly in $W^{1,2}(\Omega;\R^3)$ and strongly in $L^p(\Omega;\R^3)$ for every $p<\infty$, with $d_1$ independent of $x_2$ and $d_1=y'$ almost everywhere in $I$.

Still from \eqref{l211} and \eqref{l210} we deduce that, up to subsequences, $\e^{-1}\d_2 y_\e\weak d_2$
weakly in $W^{1,2}(\Omega;\R^3)$ and strongly in $L^p(\Omega;\R^3)$ for every $p<\infty$, with $d_2$ independent of $x_2$. Now, by passing to the limit in \eqref{ISO} we find $d_\alpha\cdot d_\beta=D_\alpha\cdot D_\beta$.

Since
$$
n_{y_\e}= \frac{\d_1 y_\e\wedge \e^{-1}\d_2 y_\e}{|\d_1 y_\e\wedge \e^{-1}\d_2 y_\e|}= \frac{\d_1 y_\e\wedge \e^{-1}\d_2 y_\e}{|D_1^\e\wedge D_2^\e|}
$$ we have that $n_{y_\e,\e}\to d_3$ in $L^p(S;\R^3)$ for every $p<\infty$,
where $d_3=(d_1\wedge d_2)/|d_1\wedge d_2|$.

The constraint \eqref{geodesic} follows from the fact that the geodesic curvature is intrinsic, i.e., the geodesic curvatures of two isometric curves are equal, see \cite{Spivak}, that is
$$
\frac{\d_1\d_1y_\e\cdot (n_{y_\e}\wedge\d_1y_\e)}{|\d_1y_\e|^3}=
\frac{\d_1\d_1\chi_\e\cdot (e_3\wedge\d_1\chi_\e)}{|\d_1\chi_\e|^3}.
$$
Rearranging and passing to the limit we find
$$
\d_1 d_1\cdot (d_3\wedge d_1)=
\frac{\d_1D_1\cdot (e_3\wedge D_1)}{|D_1|^3}|d_1|^3,
$$
and the equality \eqref{geodesic} follows since $|D_1|=|d_1|=1$.

Finally, up to subsequences, we have that $A_{y_\e,\e}$ weakly  converges to a matrix field $A$ in $L^2(S;\R^{2\times2}_{\sym})$. By using the convergences established above,  it follows that
$e_1\cdot Ae_1=y''\cdot d_3$ and
$e_1\cdot Ae_2 =d_2'\cdot d_3$. The entry $e_2\cdot Ae_2$ cannot be identified in terms of
$y$, $d_2$, and $d_3$ and is set equal to $\gamma$ in the statement.
\end{proof}

The vector fields $d_1$, $d_2,$ and $d_3$ are usually called directors: $d_1$ is tangent to the deformation $y$, $d_2$ represents the ``transversal'' orientation of the strip, and $d_3$ is orthogonal to $d_1$ and $d_2$.
The limiting values of the $11$ and $12$ components of the second fundamental form are measures of flexure and twist, respectively, cf.\ \cite{Antman}.  We also note that the constraint $\det A_{y_{\e},\e}=0$, which holds for every $\e$, does not pass to the limit. Indeed, the limit matrix field $A$ in \eqref{conv_second} may have determinant different from zero. The constraint in \eqref{constraints} asserts that the limiting beam is inextensible, while \eqref{geodesic} asserts that the limiting beam has the same geodesic curvature of the reference.

In order to state the $\Gamma$-convergence result we first introduce some definitions. We set
\begin{multline*}
\mathcal{A}:=\Big\{(d_1,d_2,d_3)\in W^{1, 2}(I; \R^{3\times3}): \  d_\alpha\cdot d_\beta=D_\alpha\cdot D_\beta, \ d_3=\frac{d_1\wedge d_2}{|d_1\wedge d_2|},
\\
\mbox{ and }d_1'\cdot (d_3\wedge d_1) = D_1'\cdot(e_3\wedge D_1) \mbox{ a.e.\ in }I\Big\},
\end{multline*}
and
$$
Q(M):=\frac12\K M\cdot M.
$$
By means of this quadratic energy density we define the constants
$$
\alpha_\K^+:=\sup\{ \alpha>0: \ Q(M)+\alpha\det M\geq0 \text{ for every } M\in \R^{2\times 2}_{\sym}\}
$$
and
$$
\alpha_\K^-:=\sup\{ \alpha>0: \ Q(M)-\alpha\det M\geq0 \text{ for every } M\in \R^{2\times 2}_{\sym}\}.
$$
The limiting energy density is the function $\overline Q:I\times\R\times \R\to [0,+\infty)$ defined by
\begin{align*}
\overline Q(x_1,\mu,\tau) : = \min \Big\{  Q(D(x_1)^{-T}(A-A^{\circ}(x_1))D(x_1)^{-1})\det D(x_1)
+\alpha^+_\K \frac{(\det A)^+}{\det D(x_1)} +\alpha^-_\K \frac{(\det A)^-}{\det D(x_1)}\ :&\\
A=
\begin{pmatrix}
\mu & \tau
\\
\tau & \gamma
\end{pmatrix}, \ \gamma\in\R
&\Big\}
\end{align*}
for every $x_1\in I$, $\mu,\tau\in\R$, where $(\det A)^+:=\det A \vee 0$, $(\det A)^-:=-(\det A\wedge 0)$, and $D(x_1)=\nabla\chi(x_1,0)$.
The $\Gamma$-limit functional $J:\mathcal{A}\to \R$ is given by
$$
J(d_1,d_2,d_3):=\int_I \o Q(x_1,d_1'\cdot d_3, d_2'\cdot d_3)\,dx_1
$$
for every $(d_1,d_2,d_3)\in\mathcal A$.

\begin{remark}
Let $D^\alpha:=D^{-T}e_\alpha$ be the contravariant vectors in the reference configurations, i.e.,
$D^\alpha\cdot D_\beta=\delta_{\alpha\beta}$. It is easy to see that $\overline Q$ has also the following characterization:
\begin{align*}
\overline Q(x_1,\mu,\tau) : = \min \Big\{  \big(Q(M-& D^{-T}A^{\circ}D^{-1})+\alpha^+_\K{(\det M)^+}+\alpha^-_\K {(\det M)^-}\ \big)\det D:\\
& M=\mu D^1\otimes D^1+\tau(D^1\otimes D^2+D^2\otimes D^1)+\gamma D^2\otimes D^2, \ \gamma\in\R
\Big\}.
\end{align*}
\end{remark}

We are now in a position to state the $\Gamma$-convergence result.

\begin{theorem}\label{Gamma}
As $\e\to 0$, the sequence $(J_\e)$ $\Gamma$-converges to the functional $J$ in the following sense:
\begin{enumerate}
\item[(i)] {\rm (liminf inequality)} for every sequence $(y_{\e})\subset W^{2,2}_{{\rm iso}, \e}(\Omega;\R^3)$,
$y\in W^{2,2}(I; \R^3)$, and $(d_1,d_2,d_3)\in \mathcal{A}$ such that $y'=d_1$ a.e.\ in $I$, $y_{\e}\weak y$ in $W^{2,2}(\Omega; \R^3)$,
$\partial_1 y_{\e} \weak d_1$ and $\frac{\partial_2 y_{\e}}{\e} \weak d_2$ in $W^{1,2}(\Omega; \R^{3})$,
we have that
$$
\liminf_{\e\to 0} J_\e(y_\e)\ge J(d_1,d_2,d_3);
$$
\item[(ii)] {\rm (recovery sequence)} for every $(d_1,d_2,d_3)\in \mathcal{A}$ there exists a sequence
$(y_{\e})\subset W^{2,2}_{{\rm iso}, \e}(\Omega;\R^3)$ such that
$y_{\e}\weak y$ in $W^{2,2}(\Omega; \R^3)$,
$\partial_1 y_{\e} \weak d_1$ and $\frac{\partial_2 y_{\e}}{\e} \weak d_2$ in $W^{1,2}(\Omega; \R^{3})$, and
$$
\limsup_{\e \to 0} J_\e(y_\e)\le J(d_1,d_2,d_3),
$$
where $y$ is defined up to a constant by $y'=d_1$ a.e.\ in $I$.
\end{enumerate}
\end{theorem}

Theorem~\ref{Gamma} will be proved in Section~\ref{proofs}. The proof will be based on two main ingredients: a relaxation
result, which is the subject of the next section, and a geometric construction of isometric immersions done in Section~\ref{Pro}.

We conclude this section with some examples.
By the assumptions made on the tensor $\K$, in a fixed orthonormal basis we may write
$$\frac 12 \K M\cdot M=\frac 12 \K_{\alpha\beta\gamma\delta}M_{\alpha\beta}M_{\gamma\delta}=
\frac 12 \left( \begin{array}{ccc}
\K_{1111} & \K_{1122} & \K_{1112} \\
\K_{1122} & \K_{2222} & \K_{1222} \\
\K_{1112} & \K_{1222} & \K_{1212}
\end{array}\right)
\left( \begin{array}{c}
M_{11} \\
M_{22} \\
2M_{12}
\end{array}\right)
\cdot
\left( \begin{array}{c}
M_{11} \\
M_{22} \\
2M_{12}
\end{array}\right).
$$

\begin{example}\label{exorth} We consider an orthotropic material with respect to the chosen axes, i.e.,
we assume $\K_{1112}=\K_{1222}=0$. We set $2K_{11}=\K_{1111}, 2K_{12}=\K_{1122}, 2K_{22}=\K_{2222},$ and
$2K_{33}=\K_{1212}.$   Then, setting $m=(M_{11}, M_{22}, 2M_{12})^T\in\R^3$, we have
$$
Q(M)\pm\alpha \det M=(\CC\pm\alpha \DD )m\cdot m,
$$
where
$$
\CC\pm\alpha \DD=
\left( \begin{array}{ccc}
K_{11} & K_{12} & 0 \\
K_{12} & K_{22} & 0 \\
0 & 0 & K_{33}
\end{array}\right)
\pm \alpha
\left( \begin{array}{ccc}
0 & \frac12 & 0 \\
\frac12 & 0 & 0 \\
0 & 0 & -\frac14
\end{array}\right).
$$
By definition, $\alpha^\pm$ is the largest value of $\alpha$ for which all the eigenvalues of $\CC\pm\alpha \DD$ are greater or equal to zero.
A simple computation shows that the eigenvalues of $\CC\pm\alpha \DD$ are
$$
K_{33}\mp \frac{\alpha}4,\qquad\frac{K_{11}+K_{22}}2- \Big(\big(\frac{K_{11}-K_{22}}2\big)^2+\big(K_{12}\pm \frac{\alpha}2\big)^2\Big)^{1/2}
$$
where we omitted the third eigenvalue since it is always positive. By imposing these expressions to be always greater or equal to zero we find
\begin{align*}
\alpha_\K^+&=\min\{4 K_{33}, 2\big( \sqrt{K_{11}K_{22}}-K_{12}\big)\},\\
\alpha_\K^-&=2\big( \sqrt{K_{11}K_{22}}+K_{12}\big).
\end{align*}
If we take $A^\circ=0$ and $D$ equal to the identity, i.e.,  $S_\eps=\Omega_\eps$,
and we assume that $\alpha_\K^+=2\big( \sqrt{K_{11}K_{22}}-K_{12}\big)$, it follows that
$$
\overline Q(x_1,\mu,\tau)  =
\left\{
\begin{array}{ll}
 \displaystyle \frac{K_{11}\mu^4+(2 K_{12}+4 K_{33})\mu^2\tau^2+K_{22}\tau^4}{\mu^2} &
  \mbox{if } \sqrt{K_{11}} \mu^2 > \sqrt{K_{22}} \tau^2,\\[2ex]
 (4K_{33}+2\sqrt{K_{11}K_{22}}+K_{12})\tau^2 & \mbox{if } \sqrt{K_{11}} \mu^2 \le \sqrt{K_{22}} \tau^2.
\end{array}
\right.
$$
\end{example}

\begin{example}\label{exsquare}
  The case $Q(M)=|M|^2$, which corresponds to the case considered in \cite{FrHoMoPa}, can be recovered by Example~\ref{exorth} by setting
$K_{11}=K_{22}=1$, $K_{12}=0$, and $K_{33}=1/2$. In this case we obtain
$$\alpha_\K^+=\alpha_\K^-=2,$$
so that
$$
Q(M)+\alpha_\K^+ (\det M)^+ + \alpha_\K^- (\det M)^- = |M|^2 + 2 |\det M|.
$$
Again, for $A^\circ=0$ and $D$ equal to the identity,
we infer
$$
\overline Q(x_1,\mu,\tau)  =
\left\{
\begin{array}{ll}
 \displaystyle \frac{(\mu^2+\tau^2)^2}{\mu^2} &
  \mbox{if }  \mu^2 >  \tau^2,\\[2ex]
 4\tau^2 & \mbox{if }  \mu^2 \le  \tau^2.
\end{array}
\right.
$$
 \end{example}

\begin{example}
  For an isotropic material
$$
Q(M)=K_\mu |M|^2 + K_\lambda (\Tr M)^2,
$$
we have
$$
\alpha_\K^+=2 K_\mu, \qquad \alpha_\K^-=2K_\mu+4 K_\lambda,
$$
as follows from Example~\ref{exorth} with $K_{11}=K_{22}= K_\mu+K_\lambda$, $K_{12}=K_\lambda$,
and $K_{33}=K_\mu/2$.
By means of the identity
$$
(\Tr M)^2 =|M|^2 +2\det M =|M|^2 +2(\det M)^+ -2(\det M)^-
$$
which holds for every $M\in \R^{2\times 2}_{\rm sym}$, we find
\begin{align*}
Q(M)+\alpha_\K^+ &(\det M)^+ + \alpha_\K^- (\det M)^- =\\
&=K_\mu |M|^2 + K_\lambda (\Tr M)^2+2 K_\mu (\det M)^+ + (2K_\mu+4 K_\lambda)(\det M)^- \\
&= (K_\mu  + K_\lambda)|M|^2+ 2(K_\mu  + K_\lambda)(\det M)^+ +2(K_\mu  + K_\lambda)(\det M)^-\\
&= (K_\mu  + K_\lambda)|M|^2+ 2(K_\mu  + K_\lambda)|\det M|.
\end{align*}
The same result can also be obtained by observing that
$Q(M)=(K_\mu  + K_\lambda)|M|^2$ for every $M$ with $\det M=0$,
and then by applying Example~\ref{exsquare}.
\end{example}

\section{Relaxation of quadratic functionals with a determinant constraint}\label{sec_relax}

Let $\mathcal B$ be a bounded open subset of $\R^n$.
Let $z:\mathcal B\to\R$ be a measurable function and let $Q:\mathcal B\times \R^{2\times 2}_{\sym}\to[0,+\infty)$ be measurable in the first variable and quadratic in the second. Define the functional
$$
\FF : L^2\left(\mathcal B;\R^{2\times 2}_{\sym}\right)\to[0,+\infty]
$$
by
$$
\FF(M) :=
\begin{cases}
\displaystyle\int_{\mathcal B} Q(x,M(x))\, dx & \text{ if $\det M = z$ a.e.\ in } \mathcal B,
\smallskip
\\
+ \infty & \text{ otherwise.}
\end{cases}
$$

\begin{proposition}\label{lsh}
The weak-$L^2$ lower semicontinuous envelope of $\FF$ is the functional
$$
\o\FF : L^2\left(\mathcal B;\R^{2\times 2}_{\sym}\right)\to[0,+\infty)
$$
given by
$$
\o\FF(M) = \int_{\mathcal B} \left( Q(x,M(x)) + \alpha^+(x) (\det M(x)-z(x))^+ +\alpha^-(x) (\det M(x)-z(x))^-  \right)\, dx,
$$
where for every $x\in \mathcal B$
$$
\alpha^+(x):=\sup\{ \alpha>0: \ Q(x,M)+\alpha\det M\geq0 \text{ for every } M\in \R^{2\times 2}_{\sym}\}
$$
and
$$
\alpha^-(x):=\sup\{ \alpha>0: \ Q(x,M)-\alpha\det M\geq0 \text{ for every } M\in \R^{2\times 2}_{\sym}\}.
$$
\end{proposition}

\begin{remark}
If $Q(x,M)=|M|^2$ and $z=0$, then $\alpha^+=\alpha^-=2$, and the lower semicontinuous envelope takes the form
$$
\o\FF(M) = \int_{\mathcal B} \left( Q(M(x)) + 2|\det M(x)|  \right)\, dx
$$
for every $M\in  L^2\left(\mathcal B;\R^{2\times 2}_{\sym}\right)$, see also Example~\ref{exsquare}.
\end{remark}

\begin{proof}[Proof of Proposition~\ref{lsh}]
By \cite[Proposition~3.16]{FoLe} we have that $\overline\FF$ is also the sequentially lower semicontinuous envelope of $\FF$, that is,  the largest function below $\FF$ that is sequentially lower semicontinuous with respect to the weak-$L^2$ topology. Moreover, by \cite[Theorem~6.68]{FoLe}, the lower semicontinuous envelope of $\FF$ is given by
$$
\o\FF(M) = \int_{\mathcal B} Q^{**}_0(x,M(x)) \, dx,
$$
where for every fixed $x\in \mathcal B$ the function $Q^{**}_0(x,\cdot)$ is the bipolar function of $Q_0(x,\cdot)$ and
$Q_0:\mathcal B\times\R^{2\times 2}_{\sym}\to[0,+\infty]$ is defined by
$$
Q_0(x,M)=Q(x,M)+\chi_{\{\det=z\}}(x,M)
$$
for every $M\in \R^{2\times 2}_{\sym}$. Here $\chi_{\{\det=z\}}$ is the indicator function of the set $\{(x,M)\in \mathcal B\times\R^{2\times2}_{\sym}:\ \det M=z(x)\}$.

Hereafter, the variable $x$ will be dropped since it will be kept
fixed until the end of the proof. For instance, we shall write
$Q(M)$ in place of $Q(x,M)$.

We have to prove that
\begin{equation}\label{convex}
Q^{**}_0(M)=Q(M) + \alpha^+ (\det M-z)^+ +\alpha^- (\det M-z)^-
\end{equation}
for every $M\in \R^{2\times 2}_{\sym}$.

In the following we identify matrices $M\in \R^{2\times2}_{\sym}$ with vectors $m=(M_{11}, M_{22}, 2M_{12})^T\in\R^3$.
For every $m\in\R^3$ we define
$$
\det m:=m_1m_2-\frac14 m_3^2,
$$
so that, according to the previous identification, we have $\det M=\det m$. Finally, let $\CC\in\R^{3\times3}_{\sym}$ be such that
$$
Q(M)=\CC m\cdot m
$$
and let $f:\R^3\to[0,+\infty)$ be the function $f(m)=\CC m\cdot m + \chi_{\{\det=z\}}(m)$.
The thesis \eqref{convex} is equivalent to prove that
\begin{equation}\label{convex0}
f^{**}(m)=\CC m\cdot m + \alpha^+(\det m-z)^+ +\alpha^- (\det m-z)^-
\end{equation}
for every $m\in\R^3$.

Let
$$
\DD=\left( \begin{array}{ccc}
0 & \frac12 & 0 \\
\frac12 & 0 & 0 \\
0 & 0 & -\frac14
\end{array}\right),
$$
so that $\det m= \DD m\cdot m$. For every $\alpha\in\R$ we consider the matrices $\CC+\alpha\DD$. By definition of $\alpha^-$ and $\alpha^+$ we have that $\CC+\alpha\DD$ is positive definite for every $\alpha\in (-\alpha^-, \alpha^+)$, while for $\alpha=-\alpha^-$ and
$\alpha=\alpha^+$ some eigenvalues of $\CC+\alpha\DD$ become equal to $0$ and the matrix $\CC+\alpha\DD$ is positive semi-definite.

Since the functions $m\mapsto \CC m\cdot m +\alpha^+(\det m-z)$ and
$m\mapsto \CC m\cdot m -\alpha^-(\det m-z)$ are convex and continuous, and they are both below $f$, we deduce
\begin{eqnarray*}
f^{**}(m) & \geq & \max\big\{ \CC m\cdot m +\alpha^+(\det m-z),\, \CC m\cdot m -\alpha^-(\det m-z) \big\}
\\
& = & \CC m\cdot m + \alpha^+(\det m-z)^+ +\alpha^- (\det m-z)^-.
\end{eqnarray*}

To prove the converse inequality, we use the definition of bipolar function. Thus, we need to show that
for every $m,\xi\in\R^3$ we have
$$
m\cdot\xi - f^*(\xi) \leq \CC m\cdot m + \alpha^+(\det m-z)^+ +\alpha^- (\det m-z)^-,
$$
where $f^*$ is the polar function of $f$. Using the definition of $f^*$, the above inequality follows if we prove that,
for every $m,\xi\in\R^3$ there exists $\xi^*\in\R^3$ with $\det\xi^*=z$ such that
$$
\CC m\cdot m -m\cdot\xi + \alpha^+(\det m-z)^+ +\alpha^- (\det m-z)^-
\geq \CC\xi^*\cdot\xi^*-\xi^*\cdot\xi.
$$
This is equivalent to prove that for every $\xi\in\R^3$ the function
$$
g_\xi(m):=\CC m\cdot m -m\cdot\xi + \alpha^+(\det m-z)^+ +\alpha^- (\det m-z)^-
$$
attains its minimum at a point $\xi^*$ with $\det\xi^*=z$.

We first observe that $g_\xi$ is coercive, since $g_\xi(m)\geq\CC m\cdot m -m\cdot\xi$ for every $m\in\R^3$ and $\CC$ is positive definite.
Since $g_\xi$ is also continuous, $g_\xi$ attains its minimum on $\R^3$.

We now want to prove that there exists a minimizer with determinant equal to $z$. We will argue in the following way:
assume that there exists a minimizer $m^*$ with $\det m^*\neq z$; then we will show that we can construct $\xi^*$ such that
$\det\xi^*=z$ and $g_\xi(\xi^*)=g_\xi(m^*)$.

Let $m^*$ be a minimizer of $g_\xi$ with $\det m^*-z>0$. Then $m^*$ must be a critical point, that is, it is a solution to
$$
2(\CC+\alpha^+\DD)m^*=\xi.
$$
Since $\CC+\alpha^+\DD$ is symmetric, this implies that $\xi\in {\rm Ker}(\CC+\alpha^+\DD)^\perp$.

Let now $m^+\in {\rm Ker}(\CC+\alpha^+\DD)$ with $m^+\neq0$.
We note that $\det m^+<0$ since otherwise $(\CC+\alpha^+\DD)m^+\cdot m^+>0$.
Consider the family of vectors
$$
m_\lambda:=m^*+\lambda m^+.
$$
We observe that $\det m_0-z=\det m^*-z>0$, while $\det m_\lambda-z\simeq\lambda^2\det m^+<0$ for $\lambda$ large enough.
Thus, there exists a suitable $\overline\lambda>0$ for which $\det m_{\overline\lambda}=z$. We set $\xi^*=m_{\overline\lambda}$ {and we have
\begin{eqnarray*}
g_\xi(\xi^*) & = & (\CC+\alpha^+\DD)(m^*+\overline\lambda m^+)\cdot(m^*+\overline\lambda m^+)-(m^*+\overline\lambda m^+)\cdot\xi-\alpha^+z
\\
& = & g_\xi(m^*)+ \overline\lambda^2 (\CC+\alpha^+\DD)m^+\cdot  m^+ +2\overline\lambda (\CC+\alpha^+\DD)m^+\cdot m^*
-\overline\lambda m^+\cdot\xi
\\
& = & g_\xi(m^*),
\end{eqnarray*}
where we used that $m^+\in {\rm Ker}(\CC+\alpha^+\DD)$ and $\xi\in {\rm Ker}(\CC+\alpha^+\DD)^\perp$.}
A similar argument applies to the case where $\det m^*-z<0$.
\end{proof}

\section{Curves on bendings}\label{Pro}

Throughout this section we
identify vectors $a=(a_1,a_2)\in\R^2$ with the corresponding $a=(a_1,a_2,0)\in\R^3$ and viceversa.
Accordingly, we can write $a^\perp:=(a_2,-a_1)=e_3\wedge a$.

Moreover, we will use the following definition:
if $U$ is an open subset of $\R^2$, a {\em bending} of $U$ is a map $u\in W^{1,\infty}(U;\R^3)$ with
$\D u\in O(2,3):=\{Q\in\R^{3\times2}\ :\ Q^TQ=I\}$ almost
everywhere.

In the following we consider
$B\in W^{2,\infty}(I; \R^2)$ to be
an arclength-parametrized {\em embedded} curve, i.e., $|B'| = 1$ and
the continuous extension of $B$ to $\o I$ is injective.
We set $N: =e_3\wedge B'= (B')^{\perp}$.

\begin{lemma}\label{nhood}
Let $G\in W^{1,1}(I; O(2, 3) )$, and assume that there
exist $p\in C^1(\o I; \S^1)$ and $m\in L^1(I; \R^3)$
such that
\begin{equation}
\label{le-p0}
G' = m\otimes p\mbox{ a.e.\ on }I.
\end{equation}
Assume, moreover, that $B'\cdot p \neq 0$ on $\o I$. Then the following are true:
\begin{enumerate}[(i)]
\item there exists $\eta > 0$ such that
the map $\Phi : (-\eta, \eta)\times I \to \R^2$ given by
\begin{equation}\label{le-p}
\Phi(s, t) = B(t) + s\, p^{\perp}(t)
\end{equation}
is a Bilipschitz homeomorphism onto the open set
$U = \Phi\left( (-\eta, \eta)\times I \right)$;
\item the map $u : U\to\R^3$ given by
\begin{equation}
\label{le-u}
u\left(\Phi(s, t)\right) =
\int_0^t G(\sigma)B'(\sigma)\ d\sigma + sG(t)p^{\perp}(t)
\end{equation}
is a bending of $U$. More precisely,
\begin{equation}
\label{le-1}
\D u\left(\Phi(s, t)\right) = G(t)\mbox{ for a.e.\ }(s, t)\in (-\eta, \eta)\times I.
\end{equation}
\end{enumerate}
\end{lemma}

\begin{proof}
The value of the quantity $\eta$ in this proof may change from line to line.
Clearly $\Phi$ is well-defined on all of $\R\times \o I$.
We claim that for all $\rho > 0$ small enough there exist $c$, $\eta > 0$ such that,
for all $t$, $t'\in \o I$ we have
\begin{equation}\label{cur-1}
|t-t'| \geq \rho\mbox{ and }s, s'\in [-\eta, \eta]\ \implies
|\Phi(s, t) - \Phi(s', t')|\geq c.
\end{equation}
In fact, by the hypotheses on $B$, for all $\rho > 0$ there exists $c > 0$ such
that $|B(t) - B(t')|\geq 5c$ whenever $|t - t'| \geq \rho$. Taking $\eta = c$,
the implication \eqref{cur-1} follows because $|p| = 1$.

On the other hand, since $p\in C^1(\o I;\S^1)$
and since $B\in W^{2,\infty}(I; \R^2)\subset C^1(\o I;\R^2)$, we see that
$\Phi\in C^1(\R\times\o I;\R^2)$ and we compute (with $\D\Phi = (\d_s\Phi | \d_t\Phi )$)
\begin{equation}\label{le-2}
\det(\D\Phi(s,t) ) = - p(t)\cdot B'(t) + sp^{\perp}(t)\cdot p'(t)
\text{ on }\R\times\o I.
\end{equation}
For $\eta > 0$ small enough the right-hand side differs from zero
for all $(s, t)\in [-2\eta, 2\eta]\times\o I$ because $p\cdot B'\neq 0$
and $p^{\perp}\cdot p'$ is bounded on $\o I$.
Hence by continuity
$|\det\D\Phi|$ is bounded from below by a positive constant on this set.
As $\D\Phi$ is bounded on this set, the inverse function theorem implies that
there exists $\rho > 0$ such that if $(s, t)$, $(s', t')\in [-\eta, \eta]\times\o I$ then
\begin{equation}\label{cur-2}
0<|t-t'|^2 + |s - s'|^2 \leq\rho^2 \ \implies
\Phi(s, t) \neq \Phi(s', t').
\end{equation}
Combined with \eqref{cur-1} this shows that there exists $\eta > 0$
such that $\Phi$ is injective on $\o V$, where $V = (-\eta, \eta)\times I$.
Thus by the invariance of domain theorem (cf.\ \cite[Theorem~3.30]{FonsecaGangbo}),
the set $U = \Phi(V)$ is open,
and since both $\D\Phi$ and $(\det\D\Phi)^{-1}$ are in $C^0(\o V)$,
the inverse $\Psi$ of $\Phi$ is in $C^1(\o U;\R^2)$.

Denote the right-hand side of \eqref{le-u} by $f(s, t)$ and define $u = f\circ\Psi$,
which is equivalent to \eqref{le-u}.
Since $Gp^{\perp}\in W^{1,1}(I;\R^3)$, we have that $f\in W^{1,1}(V;\R^3)$.
Since $\Psi$ is Bilipschitz, we can apply the
chain rule (cf.\ \cite[Theorem~2.2.2]{Ziemer})
to conclude that $u\in W^{1,1}(U; \R^3)$ and, using the fact that $G'p^{\perp} = 0$ by hypothesis,
that
$$
\D u(\Phi(s,t))\D\Phi(s,t) = (G(t)p^{\perp}(t))\otimes e_1 +  \left( G(t)B'(t) + sG(t)(p^{\perp})'(t)\right)\otimes e_2
= G(t)\D\Phi(s,t).
$$
Since $\D\Phi$ is invertible pointwise on $V$, formula \eqref{le-1}
follows. In particular, $u$ is a bending.
\end{proof}

Let $M\in L^2(I; \R^{2\times 2}_{\sym})$.
A frame $r\in W^{1,2}(I;SO(3))$ is said to be {\em adapted} to the pair $(B, M)$ if
$r$ solves
\begin{equation}\label{ode}
r' =
\begin{pmatrix}
0 & \kappa & \mu
\\
-\kappa & 0 & \tau
\\
-\mu & -\tau & 0
\end{pmatrix}\ r
\end{equation}
with $\kappa = B''\cdot N$ and $\tau = MB'\cdot N$ and $\mu = MB'\cdot B'$.

\begin{proposition}\label{propo}
Let $p\in C^1(\o I;\S^2)$ be such that $p\cdot B' \neq 0$ on $\o I$ and let $\la\in L^2(I)$.
Let $r\in W^{1,2}(I; SO(3))$ be a frame adapted to the pair $(B, \la p\otimes p)$
and let $y\in W^{2,2}(I; \R^3)$ satisfy $y' = r^Te_1$.
Then there exists a neighborhood $U$ of $B(\o I)$ and a bending $u\in W^{2,2}(U; \R^3)$
such that $u\circ B = y$ and $A_u\circ B = \la p\otimes p$, and
\begin{equation}
\label{propo-1}
\D u(B(t) + sp^{\perp}(t)) = (r^T(t)e_1)\otimes B'(t) + (r^T(t)e_2)\otimes N(t)
\end{equation}
for all $t\in I$ and all $|s|$ small enough.
The bending $u$ is explicitly given by formula \eqref{le-u}, where $G$ denotes
the right-hand side of \eqref{propo-1}.
\end{proposition}
\begin{proof}
As $r$ is adapted to $(B, \la p\otimes p)$, it satisfies \eqref{ode} with
$\mu = \la(B'\cdot p)^2$ and $\tau = \la(B'\cdot p)(N\cdot p)$ and $\kappa = B''\cdot N$.
Set
$$
G = (r^Te_1)\otimes B' + (r^Te_2)\otimes N.
$$
Then, a short computation shows that
\begin{equation}
\label{G'}
G' = (r^Te_3)\otimes (\mu B' + \tau N).
\end{equation}
Since $(B'\cdot p)\tau = (N\cdot p)\mu$, we see that
$p^{\perp}\cdot \left( \mu B' + \tau N \right) = 0$. So
$p\parallel (\mu B' + \tau N)$, and therefore $G' = m\otimes p$
for some $m\in L^1(I; \R^3)$.

Lemma \ref{nhood} then shows that the map $\Phi(s, t) = B(t) + sp^{\perp}(t)$
is a Bilipschitz homeomorphism onto its image, and that $u$ given by \eqref{le-u}
satisfies \eqref{le-1}. In particular, $\D u\circ \Phi = G$, which is \eqref{propo-1}. Moreover,
denoting by $n$ the normal to $u$,  we have $n\circ B = r^T e_3$.
After a possible translation we also have $u\circ B = y$.

Finally,
taking derivatives in \eqref{propo-1}, recalling that for an isometric immersion $u$ the relation
$\nabla^2u_k=A_u n_k$ holds,    and using \eqref{G'}, we have
$$
(n\circ B)\otimes (A_u\circ B) B' =  (\D^2 u\circ B)B' = (r^Te_3)\otimes (\mu B' + \tau N).
$$
Inserting the definitions of $\mu$ and of $\tau$,
we see that
\begin{equation}
\label{propo-2}
(A_u\circ B)B' = (\la p\otimes p)B'.
\end{equation}
Since $A_u$ is symmetric with $\det A_u = 0$ and since $p\cdot B' \neq 0$, this readily
implies that $A_u\circ B = \la p\otimes p$.

The proof is essentially complete. However, $\Phi((-\eta, \eta)\times (0, T))$
is not a neighbourhood of $B(\o I)$, although it is a neighbourhood of $B(\o J)$ for any
subinterval $J$ of $I$ with $\o J\subset I$. So we
extend $\mu$, $\tau$ and $\kappa$ by zero to $\R$, and then we extend $B$ and $r$
by solving the Frenet equations and the system \eqref{ode}, respectively. Then there is an open interval $I_1$ with $\o I\subset I_1$
such that the hypotheses of the proposition are still satisfied on $I_1$. Applying the
preceding proof to $I_1$ leads therefore to the conclusion.
\end{proof}

\begin{remark}
In the particular case $B(t) = te_1$ and in the presence of enough regularity,
Proposition \ref{propo} and Lemma \ref{nhood} reduce to
\cite[Lemma~4.3]{FrMoPa} with
$\beta = y$ and $\kappa = 0$. Since $\kappa = 0$,
the condition $y'' \neq 0$ is equivalent to $B'\cdot p \neq 0$.
\end{remark}
\begin{remark}
Condition \eqref{le-p0} is clearly
necessary for \eqref{le-1} to hold (even for $s = 0$). In fact, \eqref{le-1}
implies
$$
G'_{i\alpha}(t) = \sum_\beta \d_\alpha\d_\beta u_i(B(t))B'_\beta(t).
$$
If $u$ is a bending, then $\d_\alpha\d_\beta u \parallel n$ for $\alpha,\beta = 1, 2$. So
indeed the range of $G'(t)$ is contained in the span of $n(B(t))$.
\end{remark}

The next lemma is a smooth approximation result within the class
of symmetric rank-one matrix fields.

\begin{lemma}\label{smoothing}
Let $M\in L^2(I; \R^{2\times 2}_{\sym})$ such that $\det M = 0$
almost everywhere on $I$. Then there exist
$p_n\in W^{1, \infty}(I, \S^1)$ and $\la_n\in C^{\infty}(\o I)$ such
that $p_n\cdot B' > 0$ on $\o I$ and
$$
\la_n p_n\otimes p_n \to M \mbox{ strongly in }L^2(I, \R^{2\times 2}).
$$
More precisely, there exist $\p_n\in C^{\infty}(\o I; (-\pi, \pi))$
such that $p_n = e^{i\p_n}B'$,
where $e^{i\p}$ denotes counter-clockwise
rotation by $\p$.
\end{lemma}
\begin{proof}
Define $p\in L^{\infty}(I; \R^2)$ by setting
$$
p :=
\begin{cases}
\frac{MB'}{|MB'|} &\mbox{ if }MB'\neq 0,
\\
N  &\mbox{ if }MB' = 0.
\end{cases}
$$
and set $\la = \Tr M$. Since $M$ is symmetric, its range is orthogonal to its kernel.
Hence
\begin{equation}\label{lesup-1}
M = \la p\otimes p.
\end{equation}
In fact, if $MB'\neq 0$ then we compute
\begin{align*}
\la (p\otimes p)B' &= (\Tr M) (p\cdot B')\ p
\\
&= \frac{(MB'\cdot B')^2 + (MB'\cdot B') (MN\cdot N)}{(MB'\cdot B')^2  + (MB'\cdot N)^2}\ MB'
\\
&= MB',
\end{align*}
where we have used the fact that $(MB'\cdot N)^2 = (MB'\cdot B')(MN\cdot N)$ because $\det M = 0$.
The above equality remains true when $MB' = 0$. Since clearly the trace of $M$ agrees
with that of $\la p\otimes p$, it follows that their $(N, N)$-components agree as well,
and \eqref{lesup-1} follows.

For fixed $\Lambda > 0$ we can consider the truncated functions
$\tilde{\la}_{\Lambda} = (\Lambda\wedge\la)\vee (-\Lambda)$. Then clearly
$$
\tilde{\la}_{\Lambda} p\otimes p \to \la p\otimes p = M
$$
in $L^2(I; \R^{2\times 2}_{\sym})$, as $\Lambda \uparrow \infty$. Hence, by taking diagonal sequences
we may assume without loss of generality that $\la\in L^{\infty}(I)$.

After possibly replacing $p$ by
$$
\t p :=
\begin{cases}
\sgn(p\cdot B')\, p &\mbox{ if }p\cdot B' \neq 0,
\\
N &\mbox{ if }p\cdot B' = 0,
\end{cases}
$$
we may assume without loss of generality that
there exists a lifting $\p\in L^{\infty}(I; (-\pi, \pi])$ such that
$
p = e^{i\p} B'.
$
Set
$$
\tilde{\p}_n := ((\pi - \tfrac{1}{n})\wedge\p)\vee (\tfrac{1}{n} - \pi)
$$
and extend $\t\p_n$ by zero to $\R$. Denote by $\p_n$ the mollification
of $\t\p_n$ on a scale $1/n$. Then $\p_n\in C^{\infty}(\o I)$ attains values in
$(-\pi, \pi)$ and $\p_n\to\p$ in $L^q(I)$ for all $q \geq 1$.

Choosing $\la_n\in C^{\infty}(\o I)$ such that $\la_n\to \la$ in $L^2(I)$,
the claim follows, because $e^{i\p_n}B'\to p$ in all $L^q(I;\R^2)$.
\end{proof}

\section{Proof of the $\Gamma$-convergence result}\label{proofs}

In this section we prove Theorem~\ref{Gamma}.

\begin{proof}[Proof of Theorem~\ref{Gamma}--(i)]
We may suppose that $\liminf_{\e\to 0} J_\e(y_\e)<\infty$, since otherwise there is nothing to prove. Then, by passing to a subsequence, we may suppose that $\limsup_{\e\to 0} J_\e(y_\e)<\infty$.
By Lemma~\ref{compactness} we have that
\begin{equation}
    \label{ayeaw}
A_{y_\e,\e}\weak A\ \mbox{ in }L^{2}(\Omega; \R^{2\times 2}_{\sym}),
\end{equation}
where
$$
A=
\begin{pmatrix}
d'_1\cdot d_3 & d_2'\cdot d_3
\\
d_2'\cdot d_3 & \gamma
\end{pmatrix}
$$
with $\gamma\in L^2(\Omega)$. We note that, after setting
$$
Q(M):=\frac12\K M\cdot M,\quad M_\e:=(D^\e)^{-T}A_{y_\e,\e}(D^\e)^{-1}\sqrt{\det D^\e},
\quad M^{\circ}_\e:=(D^\e)^{-T} A^{\circ}_\eps(D^\e)^{-1}\sqrt{\det D^\e},
$$
we have that
$$
J_\e(y_\e)=\int_\Omega Q(M_\e-M^{\circ}_\e)\, dx=\int_\Omega Q(M_\e)-\K M_\e\cdot M^{\circ}_\e+Q(M^{\circ}_\e)\, dx.
$$
By \eqref{dedu}, \eqref{wconvAo}, and \eqref{ayeaw}, we have that
$$
M_\e\weak D^{-T}AD^{-1}\sqrt{\det D}=:M,\qquad M^{\circ}_\e\to D^{-T} A^{\circ}D^{-1}\sqrt{\det D}=:M^{\circ}
$$
in $L^2(\Omega;\R^{2\times2})$. Since $\det M_\e=0$ a.e.\ in $\Omega$, we may apply Proposition~\ref{lsh} with $\mathcal B=\Omega$, and obtain
\begin{eqnarray*}
\liminf_{\e\to0} J_\e(y_\e)&\ge& \int_\Omega Q(M)+\alpha_\K^+ (\det M)^+ +\alpha_\K^- (\det M)^--\K M\cdot M^{\circ}+Q(M^{\circ})\, dx\\
&=& \int_\Omega Q(M-M^{\circ})+\alpha_\K^+ (\det M)^+ +\alpha_\K^- (\det M)^-\, dx\\
&=&\int_\Omega Q(D^{-T}(A-A^{\circ})D^{-1})\det D+{\alpha_\K^+} \frac{(\det A)^+}{\det D} +{\alpha_\K^-} \frac{(\det A)^-}{\det D}\, dx
\\
&\ge&\int_I \o Q(x_1,d_1'\cdot d_3, d_2'\cdot d_3)\,dx_1,
\end{eqnarray*}
where the last inequality follows from the definition of $\overline Q$.
This proves the liminf inequality.
\end{proof}

\begin{proof}[Proof of Theorem~\ref{Gamma}--(ii)]
Let $(d_1,d_2,d_3)\in \mathcal{A}$ and let $y\in W^{2,2}(I; \R^3)$ be such that
$y'=d_1$ a.e.\ in $I$. We set
$$\mu: = d_1'\cdot d_3 = y''\cdot d_3, \quad \tau := d_2'\cdot d_3, \quad\text{and}\quad\kappa := d_1'\cdot (d_3\wedge d_1)=D_1'\cdot (e_3\wedge D_1).$$

Let $D^\alpha:=D^{-T}e_\alpha$ be the contravariant vectors in the reference configurations, i.e.,
$D^\alpha\cdot D_\beta=\delta_{\alpha\beta}$, and let
$$
M:=\mu D^1\otimes D^1+\tau(D^1\otimes D^2+D^2\otimes D^1)+\gamma D^2\otimes D^2,
$$
where $\gamma\in L^2(I)$ is chosen so that
$$
\o Q(x_1,\mu,\tau) = \big(Q(M- D^{-T}A^{\circ}D^{-1})+\alpha^+_\K{(\det M)^+}+\alpha^-_\K {(\det M)^-}\big)\det D.
$$
The fact that $\gamma$ belongs to $L^2(I)$ follows immediately by choosing
$\mu D^1\otimes D^1+\tau(D^1\otimes D^2+D^2\otimes D^1)$ as a competitor in the definition of $\o Q$
and by using the positive definiteness of $Q$.

By Proposition~\ref{lsh}, with $\mathcal B=I$, there exists $\tilde M^{\delta}\in L^2(I; \R^{2\times 2}_{\sym})$
with $\det \tilde M^{\delta} = 0$ and such that
$\tilde M^{\delta}\weak M\sqrt{\det D}$ weakly in $L^2(I;\R^{2\times2}_{\sym})$ and
\begin{align*}
\int_IQ(\tilde M^{\delta})\,dx_1&\to \int_I \left( Q(M\sqrt{\det D}) + \alpha^+_\K (\det (M\sqrt{\det D}))^+ +\alpha^-_\K (\det (M\sqrt{\det D}))^- \right)\, dx_1\\
&=\int_I \left( Q(M) + \alpha^+_\K (\det M)^+ +\alpha^-_\K (\det M)^-  \right)\det D\, dx_1.
\end{align*}
Let $M^\delta=\tilde M^\delta/\sqrt{\det D}$. Then $\det M^{\delta} = 0$ and
$M^{\delta}\weak M$ weakly in $L^2(I;\R^{2\times2}_{\sym})$ and
\begin{align}
\int_IQ(M^{\delta}- D^{-T}A^{\circ}D^{-1})\det D\,dx_1&=\int_IQ(\tilde M^{\delta})-(\K D^{-T}A^{\circ}D^{-1}\cdot M^{\delta}-Q(D^{-T}A^{\circ}D^{-1}))\det D\,dx_1\nonumber
\\
&\to J(d_1,d_2,d_3).\label{qmdlim}
\end{align}
By Lemma \ref{smoothing} with $B(t) := \chi(t,0)$, hence $B'=D_1$,
we may assume without loss of generality that there exist
$\la^{\delta}\in C^{\infty}(\o I)$ and $p^{\delta}\in C^{1}( I, \S^1)$ (same regularity of $B'$) such that
 $p^{\delta} \cdot D_1>0$ on $\o I $ and
$$
M^{\delta} = \la^{\delta} p^{\delta}\otimes p^{\delta}.
$$
We let
$r^{\delta}\in W^{1,2}(I; SO(3))$ be a frame adapted to the pair $(B,M^\delta)$, i.e.,
\begin{equation}\label{ODEdelta}
(r^{\delta})'=\begin{pmatrix}
0 & \kappa^{\delta}_M & \mu^{\delta}_M
\\
-\kappa^{\delta}_M & 0 & \tau^{\delta}_M
\\
-\mu^{\delta}_M & -\tau^{\delta}_M & 0
\end{pmatrix}\ r^{\delta}
\end{equation}
with $\kappa^{\delta}_M = D_1'\cdot (e_3\wedge D_1)$ and $\tau^{\delta}_M = M^{\delta}D_1\cdot (e_3\wedge D_1)$ and $\mu^{\delta}_M = M^{\delta}D_1\cdot D_1$. We take $r^{\delta}(0)=(d_1|d_3\wedge d_1|d_3)^T(0)$ as initial condition. Finally, we define $d_1^{\delta} := (r^{\delta})^Te_1$ and
$$
\beta^{\delta}(t) := y(0) + \int_0^t d_1^{\delta}(s)\, ds.
$$
For each $\delta > 0$,  Proposition \ref{propo} yields a neighbourhood $U^\delta$ of $B(\o I)$ and an isometry $u^{\delta} : U^\delta\to\R^3$
such that $u^{\delta}\circ B = \beta^{\delta}$ and
$$
(\D u^{\delta})\circ B =(r^{\delta})^Te_1\otimes D_1 + (r^{\delta})^Te_2\otimes (e_3\wedge D_1),
$$
and $(A_{u^{\delta}})\circ B = M^{\delta}$.

We let
$r\in W^{1,2}(I; SO(3))$ be a frame adapted to the pair $(B,M)$, i.e.,
$r$ satisfies \eqref{ode} with $\kappa, \tau,$ and $\mu$ replaced by
 $\kappa_M = D_1'\cdot (e_3\wedge D_1)$, $\tau_M = MD_1\cdot (e_3\wedge D_1)$, and $\mu_M = MD_1\cdot D_1$, respectively. Again, we take $r(0)=(d_1,d_3\wedge d_1,d_3)^T(0)$ as initial condition.
Since $M^{\delta}\weak M$ weakly in $L^2(I;\R^{2\times2}_{\sym})$ we have that $\mu^{\delta}_M\weak\mu_M$ and $\tau^{\delta}_M\weak\tau_M$ weakly in $L^2(I)$. Thus, $r^{\delta}\weak r$ weakly in $W^{1,2}(I;SO(3))$.
To identify $r$ note that $\kappa_M=\kappa$ and $\mu_M=\mu$. Also, since $D^1=-e_3\wedge D_2 /|D_1\wedge D_2|$ and $D^2=e_3\wedge D_1 /|D_1\wedge D_2|$, we have
$$
\tau_M = \frac{-\mu D_1\cdot D_2 + \tau}{|D_1\wedge D_2|}= \frac{-(d_1'\cdot d_3) (d_1\cdot d_2) + d_2'\cdot d_3}{|D_1\wedge D_2|}.
$$
To simplify this expression we write
$$
d_2=(d_2\cdot d_1) d_1+ (d_2\cdot (d_3\wedge d_1)) d_3\wedge d_1=(d_2\cdot d_1) d_1+ |d_1\wedge d_2| d_3\wedge d_1,
$$
from which we deduce that $d_2'\cdot d_3=(d_2\cdot d_1) (d_1'\cdot d_3)+ |d_1\wedge d_2| (d_3\wedge d_1)'\cdot d_3$. Hence,
$$
\tau_M=\frac{|d_1\wedge d_2| (d_3\wedge d_1)'\cdot d_3}{|D_1\wedge D_2|}=-\frac{|d_1\wedge d_2| }{|D_1\wedge D_2|} d_3'\cdot (d_3\wedge d_1)= -d_3'\cdot (d_3\wedge d_1)= d_3\cdot (d_3\wedge d_1)',
$$
where we used that
\begin{equation}\label{detdet}
|D_1\wedge D_2|^2=(D_1\cdot D_1) (D_2\cdot D_2)- (D_1\cdot D_2)^2=
(d_1\cdot d_1) (d_2\cdot d_2)- (d_1\cdot d_2)^2=|d_1\wedge d_2|^2.
\end{equation}
It is now immediate to check that $r(t)=(d_1,d_3\wedge d_1,d_3)^T(t)$.

%%%%%%%%%
Thus, $r^{\delta}\weak (d_1,d_3\wedge d_1,d_3)^T$ weakly in $W^{1,2}(I;SO(3))$, and as a consequence
$\beta^{\delta}\weak y$ weakly in $W^{2,2}(I;\R^3)$ and
$$
(\D u^{\delta})\circ B \weak d_1\otimes D_1+ (d_3\wedge d_1)\otimes (e_3\wedge D_1)
$$
weakly in $W^{1,2}(I;\R^{2\times 3}).$
In particular, $((\D u^{\delta})\circ B) D_1 \weak d_1$ weakly in $W^{1,2}(I;\R^{3})$ and, using \eqref{detdet},
\begin{align*}
((\D u^{\delta})\circ B) D_2 &\weak (D_1\cdot D_2) d_1+ (e_3\wedge D_1 \cdot D_2) d_3\wedge d_1=
(d_1\cdot d_2) d_1+ |d_1\wedge d_2| d_3\wedge d_1\\
& =(d_1\cdot d_2) d_1+ (d_3\wedge d_1\cdot d_2) d_3\wedge d_1=d_2
\end{align*}
weakly in $W^{1,2}(I;\R^{3})$.
%%%
Since for $\eps$ small enough $S_\eps\subset U^\delta$ we may define
$$
y^{\delta}_{\e} = u^{\delta}\circ \chi_\e.
$$
The map
$$
(s, t)\mapsto \chi(t,0)+ s(p^{\delta})^{\perp}(t)
$$
is a $C^1$ diffeomorphism and, from \eqref{propo-1} and the regularity of $r^\delta$ as a solution of \eqref{ODEdelta}, we see that $u^{\delta}$ is $C^2$. Hence, as $\eps\to 0$, we have
$y^\delta_\e\to u^\delta\circ B=\beta^\delta$ in $W^{2,2}(I;\R^3)$ and, see \eqref{dyu},
$\partial_1 y^\delta_\e\to ((\nabla u^\delta)\circ B)D_1$ and $\partial_2 y^\delta_\e/\e\to ((\nabla u^\delta)\circ B)D_2$ in $W^{1,2}(I;\R^3)$. Also
\begin{align*}
%\lim_{\e\to 0}
\int_{S} Q((D^\e)^{-T}(A_{y^\delta_\e,\e}-A^\circ_\eps)(D^\e)^{-1}) \det D^\e\,dx
&= %\lim_{\e\to 0}
\int_{S} Q(A_{u^\delta}\circ \chi_\e-(D^\e)^{-T}A^\circ_\e(D^\e)^{-1}) \det D^\e\,dx\\
&\to \int_I Q(A_{u^{\delta}}\circ B-(D)^{-T}A^\circ(D)^{-1}) \det D\,dx_1\\
&= \int_I Q(M^{\delta}-(D)^{-T}A^\circ(D)^{-1}) \det D\,dx_1
\end{align*}
where to obtain the first equality we used \eqref{AyAu}. Hence, by \eqref{qmdlim} it follows that
$$
\lim_{\delta\to 0}\lim_{\e\to 0}\int_{S} Q((D^\e)^{-T}(A_{y^\delta_\e,\e}-A^\circ_\eps)(D^\e)^{-1}) \det D^\e\,dx= J(d_1,d_2,d_3)
$$
and by taking a diagonal sequence we complete the proof.
\end{proof}

\bigskip

\noindent
{\bf Acknowledgements.}
PH was supported by the Deutsche Forschungsgemeinschaft.
MGM acknowledges support by GNAMPA--INdAM and by the European Research Council under Grant No.\ 290888
``Quasistatic and Dynamic Evolution Problems in Plasticity and Fracture''.

\end{document}